\documentclass[english, 12pt]{article}
\usepackage{amsfonts}
\usepackage{amsmath}
\usepackage{amssymb}
\usepackage{amsthm}
\usepackage{amscd}
\usepackage{amscd}
\usepackage{amsthm}

\newtheorem{theo}{Theorem}[section]

\newtheorem{lem}{Lemma}[section]

\newtheorem{rk}{Remark}[section]

\newtheorem{coro}{Corollary}[section]
\numberwithin{equation}{section}
\def\il{\int}

\def\calA{{\cal{A}}}

\def\R{{\mathbb{R}^d}}

\def\calE{{\cal{E}}}

\def\calC{{\cal{C}}}

\newcommand{\bb}{B_\beta}
\newcommand{\cf}{\frac{\calA{(d,\alpha)}}{2}}
\newcommand{\Wo}{W_0^{\alpha/2,2}(\Omega)}

\newcommand{\varp}{\varphi}
\newcommand{\lam}{\lambda}

\newcommand{\Om}{\Omega}

\newcommand{\alp}{\alpha}

\newcommand{\phiok}{\varp_0^{(k)}}
\newcommand{\phiv}{\varp_0^{V}}
\newcommand{\phis}{\varp_0^{*}}
\newcommand{\lamok}{\lam_0^{(k)}}
\newcommand{\lamv}{\lam_0^{V}}

\begin{document}
\bibliographystyle{alpha}

\title{Pointwise estimates for the ground states of singular Dirichlet fractional Laplacian}

\author{\normalsize Ali Beldi
, Nedra Belhajrhouma 
 \& Ali BenAmor\footnote{corresponding author} 
}

\date{}
\maketitle
\begin{abstract} We establish sharp pointwise estimates for the ground states of some singular fractional Schr\"odinger operators on relatively compact Euclidean subsets. The considered operators are of the type $(-\Delta)^{\alpha/2}|_\Om-c|x|^{-\alpha}$, where $(-\Delta)^{\alpha/2}|_\Om$ is the fraction-Laplacien
on an open subset $\Om$ in $\R$ with zero exterior condition and $0<c\leq(\frac{d-\alpha}{2})^2$. The intrinsic ultracontractivity property for such operators is discussed as well and a sharp large time asymptotic for their heat
kernels is derived.
\end{abstract}

\section{Introduction}
Let  $L_0:=(-\Delta)^{\alpha/2}|_\Om,\ 0<\alpha<\min(2,d)$ be the fractional Laplacien on an open bounded subset $\Om\subset\R$ with zero exterior condition in $L^2(\Om,dx)$. It is well known  that $L_0$ has purely discrete spectrum $0<\lam_0<\lam_1<\cdots<\lam_k\to\infty$ and that the associated semigroup $T_t:=e^{-tL_0}, t>0$ is irreducible. Hence $L_0$ has a unique strictly positive normalized ground state $\varp_0$. Moreover according to Chen--Kim--Song \cite[Eq.(4.1)]{kim2010} if $\Om$ is a $C^{1,1}$ domain (in he sense defined in \cite{kim2010}) then, $\varphi_0$ enjoys  the property of being comparable to the function $\delta^{\alpha/2}(x)$, where  $\delta(x)$ is  the Euclidian  distance function between  $x$ and the  boundary of $\Om$ which we denote by $\partial\Om$. In other words
\begin{eqnarray}
\varphi_0\sim\delta^{\alpha/2}\ {\rm on}\ \Om.
\label{similar-free}
\end{eqnarray}
Furthermore Kulczycki proved in \cite{kul} that the semigroup $T_t,\ t>0$ is intrinsically ultracontractive
(IUC for short) regardless the regularity of $\Om$. The latter property induces among others the large time
asymptotic for the heat kernel $p_t$ of  $e^{-tL_0},\ t>0$:
\begin{eqnarray}
p_t(x,y)\sim e^{-t\lam_0}\varphi_0(x)\varphi_0(y),\ {\rm on}\ \Om\times\Om.
\end{eqnarray}
Such type of estimates are very important in the sense that they give precise information on the local behavior of the ground state and the heat kernel (for large $t$) as well as on their respective rates of decay at the boundary.\\
Set $G$ the Green's Kernel of $L_0^{-1}$ and let $V\geq 0\ a.e.$ be a potential such that
\begin{eqnarray}
\int_\Om G(x,y)V(y)\,dy<1,
\end{eqnarray}
( a Kato potential for example ) and define
\begin{eqnarray}
L_V:=L_0 - V.
\end{eqnarray}
Owing to a resolvent formula it is possible to show that $L_V$ still has the same spectral properties of $L_0$. In particular it has a unique striclty positive (continuous if $V$ is Kato) and bounded ground state $\varp_0^V$. On the other side if the Green kernels $G$ of $L_0$ and $G^V$ of $L_V$ are comparable (it is the case for $C^{1,1}$ domains at least, see \cite[Proposition 9.3, Theorem 2.4]{Hansen-comp}) we conclude that $\varp_0$ and $\varp_0^V$ are comparable as well. In particular we still have
\begin{eqnarray}
\varphi_0^V\sim\varphi_0\sim\delta^{\alpha/2}\ {\rm if}\ \Om\ {\rm is}\ C^{1,1}.
\end{eqnarray}
The comparability between $\varp_0^V$ and $\varp_0$ still holds true if one replaces $V$ by a positive measure which is in the Kato class or is potentially small, i.e.,
\begin{eqnarray}
K^\mu 1:=\int_\Om G(\cdot,y)\,d\mu(y)<1,
\end{eqnarray}
provided the Green kernels $G$ and $G^\mu$ are comparable. For conditions insuring comparability of Green functions we refer the reader to \cite{Hansen-comp}.\\
Potentials of the type
\begin{eqnarray}
\frac{c}{|x|^r},\ 0\leq r<\alpha,
\end{eqnarray}
enter in the latter category.\\
However, for the limiting power $r=\alpha$ they do not fit into any of the mentioned type of potentials if $0\in\Om$. Indeed, for this case one has
\begin{eqnarray}
{\rm ess}\sup_\Om \int_\Om G(\cdot,y)|y|^{-\alp}\,dy=\infty.
\end{eqnarray}
One of our  aims in this paper is to prove that in the latter case, the ground state has singularities, describe them as well as their decay at the boundary for an open bounded subset $\Om$ containing the origin.\\
Precisely we shall prove that for
\begin{eqnarray}
V(x)=\frac{c}{|x|^\alpha},\ 0<c\leq c^*:=(\frac{d-\alpha}{2})^2,\ L_V:=L -V,
\end{eqnarray}
the operator $L_V$ still has discrete spectrum, a unique normalized ground state $\varp_0^V>0\ a.e.$ and there is $0<s\leq\frac{d-\alpha}{2}$ such that
\begin{eqnarray}
\varphi_0^V\sim\delta^{\alpha/2}|\cdot|^{-s}\ {\rm if}\ \Om\ {\rm is}\ C^{1,1}.
\end{eqnarray}
We shall however,  prove that the intrinsic ultracontractivity property is still preserved. Namely, the operator $e^{-tL_V},\ t>0$ is IUC for domains which are less regular than $C^{1,1}$ domains.\\
For $\alpha=2$ (the local case), various types of comparison results as well as pontwise estimates for ground states of the Dirichlet-Schr\"odinger operator were obtained in \cite{zuazua, dupaigne-nedev,davilla-dupaigne,cipriani-grillo,davies-book} and in \cite{benamor-osaka} for more general potentials in the framework of (strongly) local Dirichlet. Whereas the preservation of the intrinsic ultracontractivity can be found in \cite{banuelos91} for Kato potentials, in \cite{cipriani-grillo} and in \cite{benamor-osaka} in the framework of (strongly) local Dirichlet.\\
Although we shall focus on the very special case $V_c=\frac{c}{|x|^\alpha}, 0<c\leq c^*$,
our results are still valid in the following more general situations:
\begin{enumerate}
\item For positive potentials $V$ such that $V$ is bounded away from the origin and
\begin{eqnarray}
\kappa'V_{c^*}\leq V\leq \kappa V_{c^*}\ a.e.\ {\rm near}\ 0,\ \kappa\leq 1.
\end{eqnarray}
\item For signed potentials $V=V^+-V^-\in L^1_{loc}$ such that $V^-$ is of the type described  in 1.
\end{enumerate}
Our method relies  basically on an improved Sobolev inequality together with  a transformation argument (Doob's transformation) which leads to a generalized ground state representation.\\
The paper is organized as follows: In section2 we give the backgrounds together with some preparing results. For the comparability of the ground states we shall consider two situations separately: the subcritical (section3) and the critical case (section4). The last section is devoted to intrinsic ultracontractvity.
\section{Preparing results}
We first give some preliminary results that are necessary for the later development of the paper. Some of them are known. However, for the convenience of the reader we shall give new proofs for them.\\
Let $0<\alpha<\min(2,d)$. Consider the quadratic form $\calE^\alpha$ defined in $L^2:=L^2(\R,dx)$ by
\begin{eqnarray}
\calE^\alpha(f,g)&=&\frac{1}{2}\calA{(d,\alpha)}\il_{\R}\il_{\R} \frac{(f(x)-f(y))(g(x)-g(y))}{|x-y|^{d+\alpha}}\,dxdy,\nonumber\\
D(\calE^\alpha)&=&W^{\alpha/2,2}(\R)
=\{  f\in L^2(\R)\colon\,\calE^\alpha[f]:=\calE^\alpha(f,f)<\infty\},\,
\end{eqnarray}
where
\begin{eqnarray}
\calA{(d,\alpha)}=\frac{\alpha\Gamma(\frac{d+\alpha}{2})}
{2^{1-\alpha}\pi^{d/2}\Gamma(1-\frac{\alpha}{2})}.
\label{analfa}
\end{eqnarray}
It is well known that $\calE^\alpha$ is a transient Dirichlet form and  is related (via Kato representation theorem) to the selfadjoint operator, commonly named the  $\alpha$-fractional Laplacian on  $\R$ which we shall denote by  $ (-\Delta)^{\alpha/2}$.\\
Alternatively, the expression of the operator  $(-\Delta)^{\alpha/2}$ is given by (see \cite[Eq.3.11]{bogdan})
\begin{equation}\label{Lapf}
    (-\Delta)^{\alpha/2}f(x)= \calA{(d,\alpha)} \lim_{\epsilon\rightarrow 0^+}\int_{\{y\in{\mathbb{R}^d}, |y-x|>\epsilon\}} {\frac{f(x)-f(y)}{|x-y|^{d+\alpha}} dy},
\end{equation}
provided the limit exists and is finite.\\
From now on we shall ignore in the notations the dependence on $\alpha$ and shall set $\int\cdots$ as a shorthand for $\int_{\R}\cdots$. The notation q.e. means quasi everywhere with respect to the capacity induced by $\calE$.\\
For every open  subset  $\Om\subset\R,$ we  denote by $L_0:=(-\Delta)^{\alpha/2}|_\Om$ the localization of  $(-\Delta)^{\alpha/2}$ on $\Om$, i.e., the operator which Dirichlet form in $L^2(\overline\Om,dx)$ is given by
\begin{eqnarray*}
D(\calE_\Om)&=&W_0^{\alpha/2}(\Om)\colon=\{f\in W^{\alpha/2,2}(\R)\colon\, f=0 ~~~\calE-q. e.~on~\Om^c\}\nonumber\\
\calE_\Om(f,g)&=&\frac{1}{2}\calA{(d,\alpha)}\int
\int\frac{(f(x)-f(y))(g(x)-g(y))}{|x-y|^{d+\alpha}}\,dxdy\nonumber\\
&=&\frac{1}{2}\calA{(d,\alpha)}\big(\int_\Om\int_\Om \frac{(f(x)-f(y))(g(x)-g(y))}{|x-y|^{d+\alpha}}\,dxdy+\int_\Om f(x)g(x)\kappa_\Om^{(\alpha)}(x)\,dx\big),
\end{eqnarray*}
where
\begin{eqnarray}
\kappa_\Om^{(\alpha)}(x):=\calA(d,\alpha)\int_{\Om^c}\frac{1}{|x-y|^{d+\alpha}}\,dy.
\end{eqnarray}
%
The Dirichlet form $\calE_\Om$ coincides with the closure of $\calE$ restricted to $C_c^\infty(\Om)$, and is therefore regular and furthermore transcient.\\
We also recall the known fact  that $L_0$ is irreducible even when $\Om$  is disconnected \cite[p.93]{bogdan}.\\
If moreover $\Om$ is bounded, thanks to the well known Sobolev embedding,
\begin{eqnarray}
\big(\int_{\Om}|f|^{\frac{2d}{d-\alpha}}\,dx\big)^{\frac{d-\alpha}{d}}\leq
C(\Om,d,\alpha)\calE_{\Om}[f],\ \forall\,f\in\Wo,
\label{sobo-free}
\end{eqnarray}
the operator $L_0$ has compact resolvent (that we shall denote by $K:=L_0^{-1}$) which together with the irreducibility property imply that there is a unique continuous bounded, $L^2(\Om,dx)$ normalized function $\varp_0>0$ and $\lam_0>0$ such that
\begin{eqnarray}
L_0\varp_0=\lam_0\varp_0\ {\rm on}\ \Om.
\end{eqnarray}
We shall prove that this property of $L_0$ is still preserved by perturbations of the form $c|x|^{-\alpha}$. However, singularities will appear for the ground state of the perturbed operator provided $\Om$ contains the origin.\\
Owing to the sharpness of the constant (see for example \cite{yafaev})
\begin{eqnarray}
\frac{1}{c^*}:=
\frac{\Gamma^2(\frac{d-\alpha}{4})}{2^\alpha\Gamma^2(\frac{d+\alpha}{4})},
\end{eqnarray}
in the Hardy's inequality
\begin{eqnarray}
\int_{\R}|x|^{-\alp}f^2(x)\,dx\leq C_{d,\alp}\calE[f],\ \forall\,f\in W^{\alp/2,2}(\R),
\label{HI}
\end{eqnarray}
we derive that for every $0\leq c\leq c^*$, the quadratic form  $\calE_c$ defined by
\begin{eqnarray}
D(\calE_c)=W_0^{\alpha/2,2}(\Om),\ \calE_c=\calE_\Om[f]-c\int_\Om\frac{f^2(x)}{|x|^\alpha}\,dx,\ \forall\,f\in W_0^{\alpha/2,2}(\Om),
\end{eqnarray}
is positive.\\
In order to use Doob transform and  provide thereby a generalized ground state energy representation, we are going to prove that operators $(-\Delta)^{\alpha/2}-c^{-1}|x|^{-\alpha}$ have nontrivial harmonic functions. In ther words, for every $c\in(0,c^*]$ there is a function  $w>0\ a.e.$ such that $w\in L^1_{loc},\ |\cdot|^{-\alp}w\in L^1_{loc}$ and
\begin{eqnarray}
(-\Delta)^{\alpha/2}w-c^{-1}|x|^{-\alpha}w=0\ {\rm in\ the\ sense\ of\ distributions}.
\label{harmonic1}
\end{eqnarray}
This can be proved by using Fourier transform (see \cite{frank}). However, for later use (especially in Lemmata \ref{closability},\ref{closability2}) and for the sake of completeness we shall give an alternative proof which is based on potential theoretical tools.\\
In order to rewrite Eq.(\ref{harmonic1}) in a more suitable manner, we set $K_{d,\alp}:=(-\Delta)^{-\alpha/2}$, the Riesz kernel operator given by
\begin{eqnarray}
K_{d,\alp}f:=A(d,\alp)\int |\cdot-y|^{\alp-d}f(y)\,dy,
\end{eqnarray}
where
\begin{eqnarray}
A(d,\alp):=(\frac{2}{\pi})^{d/2}\frac{\Gamma(\frac{d-\alp}{2})}{2^{\alp}
\Gamma(\frac{\alp}{2})}.
\end{eqnarray}
Then Eq.(\ref{harmonic1}) is equiavlent to
\begin{eqnarray}
K_{d,\alp}(|\cdot|^{-\alp}w)(x)=cw(x)-a.e.
\label{harmonic2}
\end{eqnarray}
Clearly the potential candidates for satisfying (\ref{harmonic1})  are radially symmetric functions. Thus we shall look for $w(x)$ of the type $w(x):=|x|^{-\beta}$ where $\beta>0$.
\begin{lem} Let $w(x)=|x|^{-\beta},\ x\neq 0$. Then for every $0<\beta<d-\alp$ we have
\begin{eqnarray}
K_{d,\alp}(|\cdot|^{-\alp}w)(x)=C_{d,\alp,\beta}w(x)-a.e.,
\end{eqnarray}
where
\begin{eqnarray}
C_{d,\alp,\beta}= 2^{\alp}\frac{\Gamma(\frac{\alp+\beta}{2})
\Gamma(\frac{d-\beta}{2})}{\Gamma(\frac{d-(\alp+\beta)}{2}) \Gamma(\frac{\beta}{2})}.
\end{eqnarray}
\label{harmonic}
\end{lem}

\begin{proof} A straightforward computation yields
\begin{eqnarray*}
K_{d,\alp}(|\cdot|^{-\alp}w)(x)&=&Cw(x)-a.e.\iff\nonumber\\
w(x)&=&C^{-1}A(d,\alp)\int_{}|x-y|^{\alp-d}
|y|^{-(\alp+\beta)}\,dy\nonumber\\
&=&C^{-1}A(d,\alp)\int_{}|x-y|^{\alp-d}
|y|^{d-(\alp+\beta)-d}\,dy\nonumber\\
&=&C^{-1}A(d,\alp)\frac{A(d,d-(\alp+\beta))}{A(d,d-(\alp+\beta))}
\int_{}|x-y|^{\alp-d}
|y|^{d-(\alp+\beta)_d}\,dy\nonumber\\
&=&C^{-1}\frac{A(d,d-\beta)}{A(d,d-(\alp+\beta))}|x|^{d-\beta-d}
=C^{-1}\frac{A(d,d-\beta)}{A(d,d-(\alp+\beta))}|x|^{-\beta}-a.e..
\end{eqnarray*}
Here we have used the convolution rules for Riesz kernels (see \cite{bliedtner}) as well as the assumption $0<\beta<d-\alp$.\\
Thus
\begin{eqnarray}
C^{-1}&=&C_{d,\alp,\beta}^{-1}=\frac{A(d,d-(\alp+\beta))}{A(d,d-\beta)}\nonumber\\
&=&\frac{\Gamma(\frac{\alp+\beta}{2})}{2^{d-(\alp+\beta)}\Gamma(\frac{d-(\alp+\beta)}{2})}
\frac{2^{(d-\beta)}\Gamma(\frac{d-\beta}{2})}{\Gamma(\frac{\beta}{2})}\nonumber\\
&=&2^{\alp}\frac{\Gamma(\frac{\alp+\beta}{2}) \Gamma(\frac{d-\beta}{2})}{\Gamma(\frac{d-(\alp+\beta)}{2}) \Gamma(\frac{\beta}{2})},
\end{eqnarray}
and the proof is completed.
\end{proof}
Once we have determined the explicit expression of the constant $C_{d,\alp,\beta}$, we are going to determine its range as a function of $\beta$. We shall in fact prove that its range is $(0,c^*]$.\\
Denote by $F$, the function defined as follows:
\begin{eqnarray}
F&=&[0,d-\alp]\to\R_+\nonumber\\
&&\beta\mapsto 2^{\alp}\frac{\Gamma(\frac{\alp+\beta}{2}) \Gamma(\frac{d-\beta}{2})}{ \Gamma(\frac{\beta}{2})\Gamma(\frac{d-(\alp+\beta)}{2})},\ {\rm if}\ 0<\beta<d-\alp\ {\rm and}\
F(0)=F(d-\alp)=0.
\end{eqnarray}
Observing that $F(\beta)=F(d-\alpha-\beta)$ for very $\beta\in[0,d-\alpha]$, it suffices to precise the range of $F$ restricted to the interval $[0,\frac{d-\alpha}{2}]$.\\
Set
\begin{eqnarray}
\beta_c:=\frac{d-\alpha}{2},\ {\rm then}\  c^*=F(\beta_c)=(\frac{d-\alpha}{2})^2.
\end{eqnarray}
\begin{lem} The function $F$ is monotone increasing on the interval $[0,\beta_c]$. It follows that $F([0,\beta_c])=[0,c^*]$ and that for every $0<c\leq c^*$ there is a unique $\beta\in(0,\beta_c]$ such that $w_\beta(x)=|x|^{-\beta},\ x\neq 0$ solves Eq.(\ref{harmonic1})
\label{range}
\end{lem}
\begin{proof}
We are going to prove that the derivative of $F$ vanishes in $(0,d-\alp)$ at the only point $\beta_c$. Since $F\geq 0,\ F(0)=F(d-\alp)=0$, we derive that $F$ achieves its maximum at $\beta_c$, which together with Lemma\ref{harmonic} would complete the proof.\\
Let $\beta\in(0,d-\alp)$. A direct computation yields
\begin{eqnarray}
F'(\beta)&=0&\iff\Gamma(\frac{\beta}{2})\Gamma(\frac{d-\beta}{2})
\big[\Gamma'(\frac{\alp+\beta}{2})\Gamma(\frac{d-(\alp+\beta)}{2})
+\Gamma(\frac{\alp+\beta}{2})\Gamma'(\frac{d-(\alp+\beta)}{2})\big]\nonumber\\
&&-\Gamma(\frac{\alp+\beta}{2})\Gamma(\frac{d-(\alp+\beta)}{2})
\big[\Gamma(\frac{\beta}{2})\Gamma'(\frac{d-\beta}{2})
+\Gamma'(\frac{\beta}{2})\Gamma(\frac{d-\beta}{2})\big]=0.
\end{eqnarray}
Dividing by
$$
\Gamma(\frac{\beta}{2})\Gamma(\frac{\alp+\beta}{2})\Gamma(\frac{d-\beta}{2})
\Gamma(\frac{d-(\alp+\beta)}{2}),
$$
leads to
\begin{eqnarray}
F'(\beta)=0\iff\big[\frac{\Gamma'(\frac{\alp+\beta}{2})}{\Gamma(\frac{\alp+\beta}{2})}
+\frac{\Gamma'(\frac{d-(\alp+\beta)}{2})}{\Gamma(\frac{d-(\alp+\beta)}{2})}\big]
-\big[\frac{\Gamma'(\frac{d-\beta}{2})}{\Gamma(\frac{d-\beta}{2})}+
\frac{\Gamma'(\frac{\beta}{2})} {\Gamma(\frac{\beta}{2})}\big]=0.
\label{cond1}
\end{eqnarray}
Set
$$
\Phi(z):=\frac{\Gamma'(z)}{\Gamma(z)}\ {\rm and}\ \Psi(\beta):=\Phi(\frac{\alp+\beta}{2})-\Phi(\frac{\beta}{2}).
$$
Then equation (\ref{cond1}) is equivalent to
\begin{eqnarray*}
[\Phi(\frac{\alp+\beta}{2})-\Phi(\frac{\beta}{2})]
-[\Phi(\frac{d-\beta}{2})-\Phi(\frac{d-(\alp+\beta)}{2})]=\Psi(\beta)-\Psi(d-(\alp+\beta))
=0.
\end{eqnarray*}
Now recalling the known formula for the digamma-function $\Gamma'/\Gamma$:
\begin{eqnarray}
\Phi(\beta)=-\gamma+\int_0^1\frac{1-t^{\beta-1}}{1-t}\,dt=-\gamma
+\int_0^{\infty}\frac{e^{-t}-e^{-t\beta}}{1-e^{-t}}\,dt,
\end{eqnarray}
we obtain
\begin{eqnarray}
\Psi(\beta)&&=\int_0^{\infty}\frac{e^{-\frac{t\beta}{2}}-e^{-\frac{t(\beta+\alp)}{2}}}
{1-e^{-t}}\,dt\nonumber\\
&&=\int_0^{\infty}e^{-\frac{t\beta}{2}}\frac{1- e^{-\frac{t\alp}{2}}}{1-e^{-t}}\,dt,
\end{eqnarray}
and $\Psi'(\beta)=-\frac{\beta}{2}\Psi(\beta)<0\, \forall\,\beta\in(0,d-\alp)$. Thus for $\beta\in(0,d-\alp)$, $F'(\beta)=0$ if and only if $\beta=d-(\alp+\beta)$, yielding, $\beta=\beta_c$ which finishes the proof.
\end{proof}
\begin{rk}{\rm By the way we note that Lemma \ref{harmonic}, Lemma \ref{range} together with \cite[Theorem 1.9]{fitz00} or \cite[Theorem 3.1]{benamor-forum2012} yield a proof for Hardy's inequality (\ref{HI}) with constant larger than $\frac{1}{c^*}$.
}
\end{rk}
From now on for every $\beta\in(0,\beta_c)$, we set
\begin{eqnarray}
B_\beta:=F(\beta),\ V_\beta:=\frac{\bb}{|x|^\alpha}\ {\rm and}\ V_{c^*}:=\frac{c^*}{|x|^\alpha}.
\end{eqnarray}

\section{The subcritical case}

We fix $\beta\in(0,\beta_c)$, neglect the dependence on $\beta$ and set $V:=V_\beta$. It follows from the above considerations together with Hardy's inequality (\ref{HI}) that
\begin{eqnarray}
\int f^2(x)V(x)\,dx\leq\frac{\bb}{c*}\calE_\Om[f],\ \forall\,f\in W_0^{\alpha/2,2}(\Om).
\end{eqnarray}
Having in mind that $0<\frac{\bb}{c*}<1$, we conclude that the quadratic form which we denote by $\calE_V$ and which is defined by
\begin{eqnarray}
D(\calE_V)=W_0^{\alpha/2,2}(\Om),\ \calE_V[f]=\calE_\Om[f]-\int f^2(x)V\,dx,\ \forall\,f\in W_0^{\alpha/2,2}(\Om),
\end{eqnarray}
is closed in $L^2(\Om,dx)$ and is even comparable to $\calE_\Om$. Hence setting $L_V$ the positive selfadjoint operator associated to $\calE_V$, we conclude that $L_V$ has purely discrete spectrum $0<\lam_0^V<\lam_1^V<\cdots<\lam_k^V\to\infty$, as well.\\
Furthermore the associated semigroup $e^{-tL_V},\ t>0$ is irreducible (it has a kernel which dominates the heat kernel of the free operator $L_0$). Thereby there is a unique $\varp_0^V\in\Wo$ such that
\begin{eqnarray}
\|\varp_0^V\|_{L^2}=1,\ \varphi_0^V> 0\ q.e.\ {\rm and}\ L_V\varp_0^V=\lam_0^V\varp_0^V.
\end{eqnarray}
%
%
%
%
%
%
In the goal of obtaining the precise behavior of the ground state, we proceed to transform the form $\calE_V$ into  a Dirichlet form on $L^2(\Om,w^2dx)$, where $w(x)=|x|^{-\beta}$.\\
Let $Q$ the $w$-transform of $\calE_V$, i.e.,  the quadratic form defined in $L^2(\Om,w^2dx)$ by
\begin{eqnarray}
D(Q):=\{f\colon\,wf\in\Wo\}\subset L^2(\Om,w^2dx),\ Q[f]=\calE_V[wf],\ \forall\,f\in\,D(Q).
\end{eqnarray}
\begin{lem} The form $Q$ is a regular Dirichlet form and
\begin{eqnarray}
Q[f]=\cf\il\il \frac{(f(x)-f(y))^2}{|x-y|^{d+\alpha}} w(x)w(y)\,dxdy,\ \forall\,f\in D(Q).
\label{ID}
\end{eqnarray}
\label{closability}
\end{lem}

\begin{proof} Obviously $Q$ is closed and densely defined as it is unitary equivalent to the closed densely defined form $\calE_V$. Let us prove (\ref{ID}).\\
Writing
\begin{eqnarray}
w(x)w(y)\big(\frac{g(x)}{w(x)}-\frac{g(y)}{w(y)}\big)^2&=&(g(x)-g(y))^2+g^2(x)\frac{(w(y)-w(x))}{w(x)}\nonumber\\
&&+g^2(y) \frac{w(x)-w(y)}{w(y)},
\end{eqnarray}
and setting $g=wf$,  we get
\begin{eqnarray}
Q[f]&=&\cf\il_{}\il_{}\frac{(f(x)-f(y))^2}{|x-y|^{d+\alpha}}w(x)w(y)\,dx\,dy\nonumber\\
&+&\calA{(d,\alpha)}\il_{}\il \frac{w(x)-w(y)}{|x-y|^{d+\alpha}}{f^2(x)}{w(x)}\,dx\,dy\nonumber\\
 &-&\il{}f^2(x)w^2(x)V(x)\,dx\ \forall\,f\in {\cal{C}},
\end{eqnarray}
Having Hardy's inequality in mind and observing that
\begin{eqnarray}
Q[f]\geq\calA{(d,\alpha)}\il_{}\il \frac{w(x)-w(y)}{|x-y|^{d+\alpha}}{f^2(x)}{w(x)}\,dx\,dy\nonumber\\
 -\il{}f^2(x)w^2(x)V(x)\,dx\ \forall\,f\in D(Q),
\end{eqnarray}
we derive in particular that the integral
\begin{eqnarray}
\il_{}\il \frac{w(x)-w(y)}{|x-y|^{d+\alpha}}{f^2(x)}{w(x)}\,dx\,dy\ \ {\rm is\ finite}.
\end{eqnarray}
Thus using Fubini's together with dominated convergence theorem, we achieve
\begin{eqnarray}
Q[f]&=&\cf\il_{}\il_{}\frac{(f(x)-f(y))^2}{|x-y|^{d+\alpha}}w(x)w(y)\,dx\,dy\nonumber\\
&+& \calA{(d,\alpha)}\il_{}{f^2(x)}{w(x)}\big(\lim_{\epsilon\rightarrow0}
\il_{\{|x-y|>\epsilon\}}\frac{w(x)-w(y)}{|x-y|^{d+\alpha}}\,dy\big)\,dx\nonumber\\
 \nonumber\\
 &-&\il{}f^2(x)w^2(x)V(x)\,dx\ \forall\,f\in D(Q).
\label{AVD1}
\end{eqnarray}
Now, owing to  the fact that $w$ is a solution of the equation
\begin{eqnarray}
 (-\Delta)^{\frac{\alpha}{2}}w-Vw=0,
\end{eqnarray}
having in Eq.(\ref{Lapf}) in hands and  substituting in  (\ref{AVD1}) we get formula (\ref{ID}) from which we read that $Q$ is Markovian and and hence a Dirichlet form.\\
{\em Regularity}:  Relying on the  expression (\ref{ID}) of $Q$, we learn from \cite[Example 1.2.1.]{fuku-oshima}, that $C_c^\infty(\Om)\subset D(Q)$ if and only if
\begin{eqnarray}
J:=\int_\Om\int_\Om\frac{|x-y|^2}{|x-y|^{d+\alpha}}w(x)w(y)\,dx\,dy<\infty.
\end{eqnarray}
For $d=1$ the latter assumption is obviously verified.\\
Assume that $d\geq 2$. Set $r=2-\alpha$. Then $0<r<d$.\\
We rewrite $J$  as
\begin{eqnarray}
J&=&\int_\Om w(x)\big(\int_\Om\frac{|y|^{-r}|y|^{-\frac{d-r}{2}}}{|x-y|^{d-r}}
|y|^{\frac{d-\alp}{2}-\beta}|y|\,dy\big)\,dx\nonumber\\
&&\leq|\Om|^{1+\frac{d-\alp}{2}-\beta}\frac{1}{A(d,\alp)}\int_\Om w(x)K_{d,r}(|\cdot|^{-r}|\cdot|^{-\frac{d-r}{2}})\,dy\nonumber\\
&&\leq
\kappa\int_\Om w(x)|x|^{-\frac{d-r}{2}}<\infty,
\end{eqnarray}
where $\kappa$ is a finite constant and $J$ is finite. Here from the first to the second inequality we used Lemma \ref{harmonic}.\\
Hence from the Beurling--Deny--LeJan formula (see \cite[Theorem 3.2.1, p.108]{fuku-oshima}) together with the identity (\ref{ID}), we learn that $Q$ is regular, which completes the proof.
\end{proof}
\begin{rk}{\rm At this stage we quote that a generalized ground state representation for fractional Laplacian on the whole space was proved in \cite{frank}, using a different method from ours.
}
\end{rk}

We designate by $L^w$ the operator associated to $Q$ in the weighted Lebesgue space $L^2(\Om,w^2dx)$ and $T_t^w,\ t>0$ its semigroup. Then
\begin{eqnarray}
 L^w=w^{-1}L_Vw\ {\rm and}\ T_t^w=w^{-1}e^{-tL_V}w,\ t>0.
 \end{eqnarray}
\begin{theo} Set $r=\frac{2d}{d-\alpha}$,
\begin{eqnarray}
A:=\frac{1}{C(\Om,d,\alp)^{1/r}}(1-\frac{B_\beta}{c^*})^{1/r}
(\frac{1}{\lam_0^V}\sup_{x\in\Om}|x|^{2\beta})^{1-1/r}
~~\hbox{\rm and}~~q=2-1/r.
\end{eqnarray}
Then
\begin{eqnarray}
\parallel f^2\parallel_{ {L^q}(w^2dx)}\leq AQ[f],\ \forall\,f\in D(Q).
\label{w-sob}
\end{eqnarray}
It follows that for every $t>0$, the operator $T_t^w$ is ultracontractive.
\label{UC}
\end{theo}

\begin{proof} By H\"older's inequality, we get for every $f\in D(Q)$,
\begin{eqnarray}\label{l311}
  \int_{}w^2f^{2(2-1/r)}&\leq& \big(\int w^{2r}f^{2r}\big)^{1/r}\big(\int f^2\big)^{1-1/r}
\end{eqnarray}
Having Sobolev inequality (\ref{sobo-free}) in hands,  we obtain
\begin{eqnarray}\label{l321}
 \big(\int w^{2r}f^{2r}\big)^{1/r}\leq\frac{1}{C(\Om,d,\alp)}(1-\frac{B_\beta}{c^*})Q[f] .
\end{eqnarray}
On the other hand we have,
\begin{eqnarray}
\int g^2|x|^{2\beta}\,dx&\leq&\sup_{x\in\Om}|x|^{2\beta}\int g^2\,dx\leq \frac{1}{\lam_0^V}\sup_{x\in\Om}|x|^{2\beta}{\calE_V}[g],\,\forall\,g\in\Wo.
\end{eqnarray}
Setting $g=wf$, we get
\begin{equation}\label{1333}
\int f^2\leq \frac{1}{\lam_0^V}\sup_{x\in\Om}|x|^{2\beta}Q[f],\ \forall\,f\in \cal{C}
\end{equation}
Combining (\ref{l311}), (\ref{l321}) and (\ref{1333}), we get (\ref{w-sob}).\\
Finally, since $Q$ is a Dirichlet form, it is know that a Sobolev embedding for the domain of a Dirichlet form yields the ultracontrativity of the related semigroup ( see \cite[Theorems 4.1.2,4.1.3]{saloff-coste}), which ends the proof.
\end{proof}
Theorem \ref{UC} leads to a rough upper bound for $\phiv$, giving thereby partial information  about the singularities of $\phiv$.
\begin{coro} The following upper bound holds true
\begin{eqnarray}
\phiv\leq A(\frac{t}{2})^{-\frac{r}{r-2}}e^{t\lamv}w,\ a.e.,\ \forall\,t>0.
\end{eqnarray}
\label{ub}
\end{coro}
\begin{proof} The proof is quite standard so we omit it.
\end{proof}
As long as the distance function is involved (decay at the boundary) in any accurate description of $\phiv$, one needs among others a representation of $\phiv$ involving the Green kernel  $G$ or  $G^V$. This is established in:
\begin{lem} The following identity holds true
\begin{eqnarray}
\phiv&=&K(V\phiv)+\lamv K\phiv\ a.e.
\end{eqnarray}
\label{ground-rep}
\end{lem}
\begin{proof}
Set
\begin{eqnarray}
u=\phiv-K(V\phiv)-\lamv K\phiv.
\end{eqnarray}
Owing to the fact that $\phiv$ lies in $\Wo$ and hence lies in $L^2(Vdx)$, we obtain that the  measure   $\phiv  V$ has finite energy integral with respect to the Dirichlet form $\calE_\Om$, i.e.,
\begin{eqnarray}
\int |f\phiv V|\,dx\leq \gamma(\calE_\Om[f])^{1/2},\ \forall\,f\in C_c^\infty(\Om),
\end{eqnarray}
and therefore  $K(V\phiv)\in\Wo$. Thus $u\in\Wo$ and satisfies the identity
\begin{eqnarray}
\calE_\Om(u,g)&=&\calE_\Om(\phiv,g)-\int\varp gV\,dx
-\lamv\int\phiv g\,dx\nonumber\\
&=&\calE_{V}(\phiv,g)-\lamv\int\phiv g\,dx=0,\ \forall\,g\in\Wo.
\end{eqnarray}
Since $\calE$ is positive definite  we conclude that $u=0\, a.e.$, which yields the result.
\end{proof}
We state now the main theorem of this section.
\begin{theo}\label{main1} Let $\Om\subset\R$ be a bounded domain containing the origin and  $0< c<c^*$.
\begin{enumerate}
\item There exist finite constants $C_1^V,\, C_2^V>0$ such that
\begin{eqnarray}
C_1^Vw\leq\phiv\leq C_2^Vw\ a.e.\ {\rm near}\ 0.
\end{eqnarray}
\item Assume that $\Om$ satisfies the uniform interior ball condition (as defined in \cite{kim2010}) then there is a finite constant $C^V>0$ such that
\begin{eqnarray}
\phiv\geq C^V\delta^{\alpha/2}\ a.e.
\end{eqnarray}
\item If furthermore $\Om$ is $C^{1,1}$ then there is finite constants $C_1^V,\, C_2^V$ such that
\begin{equation}
C_1^Vw\delta^{\alpha/2}\leq\phiv\leq C_2^Vw\delta^{\alpha/2}\ a.e.
\end{equation}
\end{enumerate}
\label{comparison1}

\end{theo}
\begin{proof} 1. As $\phiv>0\ a.e.$, there is  $\epsilon>0$ and $\eta>0$ such that
\begin{eqnarray}
\phiv>\epsilon~~a.e~\hbox{in}~B_{2\eta}\subset\Om.
\end{eqnarray}
Choose  $C_\eta>0$ so  that    $ \epsilon>C_\eta r^{-s}$  for   $r\geq\eta$.\\
Let $0\leq j\leq 1$ be a mollifier  such that
\begin{eqnarray}
j\in C_c^\infty([0,\infty)),\ Supp\, j\subset [0,2]\ {\rm and}\ j=1\ {\rm on}
\ [0,1].
\end{eqnarray}
Now set $\theta$ the function defined by
\begin{eqnarray}
\theta(x)=\phiv(x)-C_\eta j(\frac{|x|}{\eta})w(x)\ a.e.
\end{eqnarray}
Then by \cite[Lemma 3.3]{frank} $\theta\in W_0^{\alpha/2,2}(\Om)$, yielding that $\theta^-\in W_0^{\alpha/2,2}(B_\eta)$.\\
Activating Sobolev inequality (\ref{sobo-free}) together with identity (\ref{Lapf}) and Lemma \ref{harmonic} and utilizing the fact that $\phiv$ is the ground state of $L_V$, we obtain:
\begin{eqnarray}
 \|(\theta^-)^2\|_{L^r}&\leq&C\big(\frac{1}{2}{\calA}(d,\alp)\int\int
 \frac{(\theta^{-}(x)-\theta^{-}(y))^2}{|x-y|^{d+\alpha}}\,dx\,dy-\int V(x){(\theta^{-})}^2(x)\,dx\big)\nonumber\\
  &\leq &-C\big(\frac{1}{2}{\calA}(d,\alp)\int\int\frac{(\theta(x)-\theta(y))
  (\theta^{-}(x)-\theta^{-}(y))}{|x-y|^{d+\alpha}}\,dx\,dy\nonumber\\
 &-&\int V(x)\theta(x)\theta^{-}(x)\,dx\big)\nonumber \\
 &= &-C\big(\frac{1}{2}{\calA}(d,\alp)\int\int\frac{(\phiv(x)-\phiv(y))(\theta^{-}(x)
 -\theta^{-}(y))}{|x-y|^{d+\alpha}}\,dx\,dy\nonumber\\
 &-&\int V(x)\phiv(x)\theta^{-}(x)\,dx\big)\nonumber\\
 &+&C_\eta C\big(\frac{1}{2}{\calA}(d,\alp)\int\int\frac{(w(x)-w(y))(\theta^{-}(x)
 -\theta^{-}(y))}{|x-y|^{d+\alpha}}\,dx\,dy\nonumber\\
 &-&\int V(x)w(x)\theta^{-}(x)\,dx\big)\nonumber\\
 &=&-C\lamv\int \phiv(x)\theta^{-}(x)\,dx +C C_\eta \int((-\Delta)^{\frac{\alpha}{2}}w(x)-V(x)w(x))\theta^{-}(x)\,dx\nonumber\\
 &=&-C\lamv\int \phiv(x)\theta^{-}(x)\,dx\leq 0.
\end{eqnarray}
In the 'passage' from the first to the second inequality, we used the fact that for any Dirichlet form ${\mathcal{D}}$ one has ${\mathcal{D}}(f^+,f^-)\leq 0$ (See \cite[Theorem 4.4-i)]{roeckner-ma}), whereas the equality before the last one is obtained with the help of the identity (\ref{Lapf}).\\
Finally we are led to  $\theta^{-}\equiv 0$ in $B_\eta$ yielding $\theta=\theta^+$ in $B_\eta$ and hence
\begin{eqnarray}
\phiv\geq C_\eta w\ a.e.\ {\rm on}\ B_\eta,
\end{eqnarray}
which together with Corollary \ref{ub} leads to $\phiv\sim w$ near the origin.\\
2. 
%
%
%
Relying on  \cite[Theorem 3.1]{kim2010} one ca prove by mean of a direct computation that for every bounded domain fulfilling the uniform interior ball condition the Green function satisfies the lower bound
\begin{eqnarray}
G(x,y)\geq C|x-y|^{-d+\alpha}\big(1\wedge\frac{\delta^{\alpha/2}(x)
\delta^{\alpha/2}(y)}{|x-y|^\alpha}\big),\ x\neq y.
\label{green-sharp0}
\end{eqnarray}
Hence
\begin{eqnarray}
G(x,y)\geq C\delta^{\alpha/2}(x)
\delta^{\alpha/2}(y).
\label{LBG}
\end{eqnarray}
New the desired lower bound for $\phiv$ is obtained by using formula (\ref{ground-rep}).\\
3. It suffices to prove that there is a finite constant $C$ such that $\phiv\leq C\delta^{\alp/2}$ away from the origin.\\
Owing to the $C^{1,1}$ regularity of $\Om$, one has by \cite[Corollary 1.2]{kim2010} the following sharp estimate for the Green kernel of $L_0$:
\begin{eqnarray}
G(x,y)\sim|x-y|^{-d+\alpha}\big(1\wedge\frac{\delta^{\alpha/2}(x)
\delta^{\alpha/2}(y)}{|x-y|^\alpha}\big),\ x\neq y.
\label{green-sharp}
\end{eqnarray}
We use once again formula (\ref{ground-rep}). Fix $\rho>0$. Let $x\in\Om$ be such that $|x|>\rho$. We decompose the integral
\begin{eqnarray}
  K\phiv &=&\int_{|x-y|<\varepsilon} G(x,y)\phiv(y)dy + \int_{|x-y|\geq\varepsilon} G(x,y)\phiv(y)dy\nonumber\\
&=\colon&  I_1+I_2.
\end{eqnarray}
If  $ |x-y|\geq\varepsilon$ then by (\ref{green-sharp}) there is a constant $0<C_{sim}<\infty$ depending solely on $\Om,\ d$ and $\alp$ such that
\begin{eqnarray}
 I_2\leq C_{sim}\varepsilon^{-d+2\alpha}\delta(x)^{\alpha/2}\int_{\Om}{\delta^{\alpha/2}(y)}\phiv(y)\,dy.
\end{eqnarray}
Observe that
\begin{eqnarray}
A_2^V:=\int_{\Om}{\delta^{\alpha/2}(y)}\phiv(y)\,dy<\infty.
\end{eqnarray}
On the other hand  owing to the fact that $\delta(x)\leq \delta(y)+|x-y|$ we obtain
\begin{eqnarray}
\big(1\wedge\frac{\delta(x)\delta(y)}{|x-y|^2}\big)\leq 2\big(1\wedge\frac{\delta(x)}{|x-y|}\big)\big(1\wedge\frac{\delta(y)}{|x-y|}\big),\ x\neq y.
\end{eqnarray}
Having the upper estimate on $\phiv$ in mind, (Corollary \ref{ub}), we therefore get
\begin{eqnarray}
I_1\leq C_{sim}c_t^V\delta^{\alpha/2}(x)\int_{|x-y|<\epsilon}
\frac{|y|^{-\beta}}{|x-y|^{d-\frac{\alpha}{2}}}\,dy.
\end{eqnarray}
Now if    $|x-y|<\varepsilon $ then
\begin{eqnarray}
|y|\geq |x|-|x-y|\geq \varrho-\varepsilon>0,
\end{eqnarray}
yielding
\begin{eqnarray}
I_1\leq C_{sim}c_t^V(\varrho-\epsilon)^{-\beta}\delta(x)^{\alpha/2}
\sup_{x\in\Om}\int_{\Om}\frac{dy}{|x-y|^{d-\frac{\alpha}{2}}},
\end{eqnarray}
with
\begin{eqnarray}
\sup_{x\in\Om}\int_{\Om}\frac{dy}{|x-y|^{d-\frac{\alpha}{2}}}<\infty.
\end{eqnarray}
Finally the term $K^V\phiv$ can be estimated from above by the same manner to obtain
\begin{eqnarray}
K^V\phiv&\leq& C_{sim}\varepsilon^{-d+2\alpha}\delta(x)^{\alpha/2}
 \int_{\Om}{\delta^{\alpha/2}(y)}\phiv(y)V(y)\,dy\\
 &+&C_{sim}c_t^V(\varrho-\epsilon)^{-2\beta}\delta(x)^{\alpha/2}
\sup_{x\in\Om}\int_{\Om}\frac{dy}{|x-y|^{d-\frac{\alpha}{2}}}\,,
\end{eqnarray}
away from the origin and
\begin{eqnarray}
A_1^V:=\int_{\Om}{\delta^{\alpha/2}(y)}\phiv(y)V(y)\,dy<\infty.
\end{eqnarray}
Matching all together yields the claim, which  completes the proof.
\end{proof}
\begin{rk}{\rm
%
%
%
%
From the proof of the latter theorem we learn that the $C^{1,1}$ assumption is involved only to determine the maximal decay rate of the ground state at the boundary. It is not clear for us whether this assumption is necessary or not.

}

\end{rk}
\section{The critical case}
The critical case differs in some respects from the subcritical one. The most apparent difference is that the critical  quadratic form is no longer closed on the starting fractional Sobolev space $\Wo$. Consequently the proof of Lemma \ref{ground-rep} is no more valid to express the ground state for the simple reason that it may not belongs to $\Wo$. We shall in fact show by the end of this section that $\phiv\not\in\Wo$ in the critical case. Also the reasoning in the proof of assertion 1. of the last theorem breaks down.\\
We shall however prove that the critical form is closable and has compact resolvent by mean of a Doob-transform. An approximation process will then lead to extend the identity of Lemma \ref{ground-rep} helping therefore to get the sharp estimate of the ground state.\\
Let $\dot\calE_*$ be the quadratic form defined by
\begin{eqnarray}
D(\dot\calE_*)=\Wo,\ \dot\calE_*[f]=\calE_\Om[f]-c^*\int \frac{f^2(x)}{|x|^\alpha}\,dx,\ \forall\,f\in W_0^{\alpha/2,2}(\Om).
\end{eqnarray}
By analogy to the subcritical case we define the $w_*$-transform of $\dot\calE_*$ which we denote by $\dot Q_*$ and is defined by
\begin{eqnarray}
D(\dot Q_*):=\{f\colon\,w_*f\in\Wo\}\subset L^2(\Om,w_*^2dx),\ \dot Q_*[f]=\dot\calE_*[w_*f],\ \forall\,f\in\,D(\dot Q_*).
\end{eqnarray}
Following the computations made in the proof of Lemma \ref{closability} we realize that $\dot Q_*$ has the following representation
\begin{eqnarray}
\dot Q_*[f]=\frac{\mathcal{A}(d,\alp)}{2}\il\il \frac{(f(x)-f(y))^2}{|x-y|^{d+\alpha}} w_*(x)w_*(y)\,dxdy,\ \forall\,f\in D(\dot Q_*).
\end{eqnarray}
\begin{lem} The form $\dot Q_*$ is closable in $L^2(\Om,w_*^2dx)$. Furthermore its closure  is a Dirichlet form in  $L^2(\Om,w_*^2dm)$.\\
It follows, in particular that $\dot\calE_*$ is closable.
\label{closability2}
\end{lem}
\begin{proof} We first mention that since $\dot\calE_*$ is densely defined then $\dot Q_*$ is densely defined as well.\\
Now we proceed to show that $\dot Q_*$ possesses a closed extension. To that end we introduce the form $\tilde Q$ defined by
\begin{eqnarray}
&D(\tilde Q)&:=\big\{f\colon\,f\in L^2(\Om,w_*^2dx),\  \il\il \frac{(f(x)-f(y))^2}{|x-y|^{d+\alpha}} w_*(x)w_*(y)\,dxdy<\infty\big\}\nonumber\\
&\tilde Q[f]&=\frac{\mathcal{A}(d,\alp)}{2}\il\il \frac{(f(x)-f(y))^2}{|x-y|^{d+\alpha}} w_*(x)w_*(y)\,dxdy,\ \forall\,f\in D(\tilde Q).
\end{eqnarray}
Arguing as in the proof of Lemma \ref{closability}, we obtain  that $C_c^{\infty}(\Om)\subset D(\tilde Q)$.\\
Hence from the Beurling--Deny-LeJan formula (see \cite[Theorem 3.2.1, p.108]{fuku-oshima}), the form $\tilde Q$ is the restriction to $C_c^{\infty}(\Om)$ of a Dirichlet form and therefore closable and Markovian. Since $D(\dot Q_*)\subset D(Q)$ we conclude that $\dot Q_*$ is closable and Markovian as well, yielding that its closure is a Dirichlet form. Now the closability of $\dot\calE_*$ is an immediate consequence of the closability of $\dot Q_*$  which finishes the proof.
\end{proof}
From now on we set $\calE_*$  the closure of $\dot\calE_*$ and $L_*$ the selfadjoint operator related to $\calE_*$, respectively $Q_*$ the closure of $\dot Q_*$ and $H_*$ its related selfadjoint operator. Finally $T_t^*:=e^{-tL_*},\ t>0$ and $S_t:=e^{-tH_*},\ t>0$. Obviously $H_*=w_*^{-1}L_*w_*$.\\
The development of this section depends heavily on the following improved Sobolev inequality due to Frank--Lieb--Seiringer \cite[Theorem 2.3]{frank}: For every $2<p<\frac{2d}{d-\alp}$ there is a constant $S_{d,\alp}(\Om)$ such that
\begin{eqnarray}
(\int |f|^p\,dx)^{2/p}\leq S_{d,\alp}(\Om)\dot\calE_*[f],\ \forall\,f\in\Wo,
\label{ISI}
\end{eqnarray}
Of course the latter inequality extends to the elements of $D(\calE_*)$ with $\dot\calE_*$ replaced by $\calE_*$.
The idea of using improved Sobolev type inequality to get estimates for  the ground state was already used in \cite{benamor-osaka, davilla-dupaigne}.
\begin{theo} For every $t>0$, the operator $S_t$ is ultracontractive. It follows that
\begin{itemize}
\item[i)]The operators $S_t,\ t>0$ and hence $T_t^*,\ t>0$ are  Hilbert-Schmidt operators and the operator $L_{*}$ has a  compact resolvent.
\label{spec-prop1}
\item[ii)] $ker(L_*-\lam_0^*)=\mathbb{R}\phis$ with $\phis> 0\ a.e.$
\item[iii)] If $\Omega$ satisfies the uniform interior ball condition then
\begin{equation}
\phis(x)\geq\big( C_G \lam_0^*\int\delta(y)^{\alpha/2}\phis(y)\,dy\big)\delta(x)^{\alpha/2},\ a.e.
\label{estimate-eigenfunction}
\end{equation}
\item[iv)] $\phis\leq C_te^{\lam_0^*t}w_*$.
\end{itemize}
\label{ground-critical}
\end{theo}
\begin{proof} The proof that $S_t,\ t>0$ is ultracontractive runs as the one corresponding to the subcritical case with the help of Lemma \ref{closability2} and inequality (\ref{ISI}) as main ingredient.\\
i)  Every ultracontractive operator has an almost everywhere bounded kernel and since $w_*\in L^2(\Om)$ one get that $S_t,\ t>0$ is a Hilbert-Schmidt operator as well as $T_t^*$ and hence $L_*$ has compact resolvent.\\
ii) Since $T_t^*,\ t>0$ has a nonnegative kernel it is irreducible and the claim follows from the well know fact that the generator of every irreducible semigroup has a nondegenerate ground state energy with a.e. nonnegative ground state.\\
iii) The fact that $T_t^*$ is a Hilbert-Schmidt operator yields that $L_*$ possesses a Greeen kernel, $G_*$ and that $G_*\geq G$. Writing
\begin{eqnarray}
\phis=\lam_0^*\int G_*(\cdot,y)\phis(y)\,dy\geq\lam_0^*\int G(\cdot,y)\phis(y)\,dy
\end{eqnarray}
and using the lower bound (\ref{LBG}) yields the result.\\
iv) Follows from the fact that $S_t,\ t>0$ is ultracontracive.
\end{proof}

Let $0<c_k\uparrow c_*$. Then $L_k:=L-\frac{c_k}{|x|^\alpha}$ increases in the strong resolvent sense to $L_*$. Since $L_*$ has compact resolvent, the latter convergence is even uniform (see \cite[Lemma 2.5]{bbb}). Thus setting $\lam_0^{(k)}$'s the ground state energy  of the $L_k$'s and $\phiok$ its associated ground state, we obtain
\begin{eqnarray}
\lam_0^{(k)}\to\lam_0^*\ {\rm and}\ \phiok\to\phis\ {\rm in}\ L^2(\Om,dx).
\label{approxi}
\end{eqnarray}
For an accurate description of the behavior of the ground state, we shall extend formula (\ref{ground-rep}) to $\phis$.
\begin{lem} We have
\begin{eqnarray}
\phis&=&K^{V_*}\phis+\lam_0^* K\phis.
\end{eqnarray}
\label{ground-rep2}
\end{lem}
\begin{proof} Use formula (\ref{ground-rep}) for $\phiok$'s, $\lamok$'s then pass to the limit and use the fact that $K$ is bounded from $L^1$ into itself.
\end{proof}
Now having assertion (iv) of Theorem \ref{ground-critical} in hands together with Lemma \ref{ground-rep2}, imitating the final part of the proof of Theorem \ref{main1} yields that there is a finite constant $C$ such that $\phis\leq C\delta^{\alpha/2}$ away from the origin, provided that $\Om$ is $C^{1,1}$. We thus achieve the following description of $\phiv$ for $C^{1,1}$ domains
\begin{eqnarray}
\phis\leq C\delta^{\alpha/2} w_*\ a.e.\ {\rm and}\ \phiv\sim\delta^{\alp/2}\ {\rm away\ from}\ 0.
\end{eqnarray}
Finally we resume.
\begin{theo} Let $\Om$ be a bounded domain containing the origin. Then
\begin{enumerate}
\item $\phis\sim w$  near the origin.
\item If furthermore $\Om$ is $C^{1,1}$ then
\begin{eqnarray}
\phis\sim\delta^{\alpha/2} w_*.
\end{eqnarray}
\end{enumerate}
\label{main-critical}
\end{theo}
\begin{proof} Assertion 2. was already proved. To prove the first assertion one uses Theorem \ref{main1}-3. together with the approximation result (\ref{approxi}).
\end{proof}
\begin{rk}{\rm
Theorem \ref{main-critical} indicates that $\phis\not\in \Wo$ and whence $D(\calE_*)$ contains strictly $\Wo$. Indeed, if $\phis\in\Wo$ then $\int |x|^{-\alp}(\phis)^2\,dx<\infty$. But for small $r>0$ we have
\begin{eqnarray}
\int |x|^{-\alp}(\phis)^2\,dx\geq C\int_{B_r}|x|^{-\alp}|x|^{-d+\alp}\,dx=\infty.
\end{eqnarray}
\label{strict}
}
\end{rk}
\section{Preservation of IUC}

In this section we shall discuss stability of the IUC property for the semigroups
\begin{eqnarray}
T_t^c:=e^{-tL_c},\ t>0,\ 0<c\leq c_*\ {\rm with}\ T_t^{c_*}=T_t^*.
\end{eqnarray}
Kulzycs proved (see \cite{kul}) that $T_t:=e^{-tL_0},\ t>0$ is IUC for every bounded domain, via Log-Sob inequality. We shall prove that this property still holds true for $T_t^*$ for a large class of bounded domains including, for instance $C^{1,1}$ domains.\\
One of the crucial ingredients for the proof is the following Hardy-type inequality which we assume its occurrence for the rest of the paper: There is a finite constant $C_H>0$ such that
\begin{eqnarray}
\int\frac{f^2(x)}{\varp_0^2(x)}\,dx\leq C_H\calE[f],\ \forall\,f\in\Wo.
\label{hardy}
\end{eqnarray}
\begin{rk}{\rm
The latter inequality holds true for bounded domains satisfying the uniform interior ball condition and  $d\geq 2,\ \alp\neq 1$. Indeed for this class of   domains we already observed that
\begin{eqnarray}
\varp_0\geq C\delta^{\alpha/2},
\label{LB}
\end{eqnarray}
whereas \cite[Corollary 2.4]{song} asserts  that if $\Om$ is a Lipschitz domain then  for every $\alp\neq 1$ and $d\geq 2$ we have
\begin{eqnarray}
\int\frac{f^2(x)}{\delta^\alpha(x)}\,dx\leq C_H\calE_\Om[f],\ \forall\,f\in\Wo,
\label{LB2}
\end{eqnarray}
Combining the two inequalities yields (\ref{hardy}).

}
\end{rk}
Hereafter we shall designate by $Q^{\phis}$ the $\phis$-transform of $\calE_*$ and $\calC^{\phis}$ its form core defined by
\begin{eqnarray}
\calC^{\phis}:=\big\{f\colon\,w_*f\in \Wo\big\}.
\end{eqnarray}
Following the spirit of the proof of Lemma \ref{closability} together with Lemma \ref{closability2}, we derive the expression of $\calE^{\phis}$ on $\calC^{\phis}$:
\begin{eqnarray}
\calE^{\phis}[f]=\il\il \frac{(f(x)-f(y))^2}{|x-y|^{d+\alpha}}\phis(x)\phis(y)\,dxdy
+\lam_0^*\int(f\phis)^2(x)\,dx
,\ \forall\,f\in\calC^{\phis}.
\label{EC}
\end{eqnarray}
We also recall the known that since $T_t$ is IUC then there is a finite constant $C_G$ such that
\begin{eqnarray}
G(x,y)\geq C_G\varp_0(x)\varp_0(y).
\label{greeno}
\end{eqnarray}

\begin{lem} There is a finite constant $\tilde{C_G}$ such that
\begin{eqnarray}
\phis\geq\tilde{C_G}\varp_0\ a.e.
\end{eqnarray}
\label{fis}
\end{lem}
\begin{proof} Lemma \ref{ground-rep2} together with the observation (\ref{greeno}) yield
\begin{eqnarray}
\phis\geq\lam_0^*K\phis\geq\lam_0^*C_G\varp_0\int\varp_0(y)\phis(y)\,dy,
\end{eqnarray}
which was to be proved.
\end{proof}

\begin{theo} The semigroups $T_t^c,\ t>0,\ 0<c\leq c_* $ are IUC.
\label{IUC}
\end{theo}
\begin{proof} We shall give the proof only  for $T_t^*$. For the other cases the proof is similar.\\
We shall show that $T_t^{\phis},\ t>0$ is ultracontractive. Various constants in this proof will be denoted simply by $C$.\\
On one hand, thanks to the improved Sobolev inequality together with formula (\ref{EC}) we obtain
\begin{eqnarray}
\big(\int(|f|\phis)^p\,dx\big)^{2/p}&\leq& C\big(\il\il \frac{(f(x)-f(y))^2}{|x-y|^{d+\alpha}}\phis(x)\phis(y)\,dxdy\nonumber\\ &+&\lam_0^*\int(f\phis)^2(x)\,dx\big),\ \forall\,f\in\calC^{\phis}.
\end{eqnarray}
Meanwhile, Lemma \ref{fis} together with the lower bound on $\varp_0$ in (\ref{LB}) lead to
\begin{eqnarray}
 \phis\geq C\varp_0,
\end{eqnarray}
which in conjunction with inequality (\ref{hardy}) leads to
\begin{eqnarray}
\int f^2(x)\,dx&\leq& C_H\calE^{\varp_0}[f]
=C\big(\int\int\frac{(f(x)-f(y))^2}{|x-y|^{d+\alpha}}\varp_0(x)\varp_0(y)\,dx\,dy
\nonumber\\
&+&\lam_0\int (f\varp_0)^2(x)\,dx\big)
\leq CQ^{\phis}[f],\ \forall\,f\in\calC^{\phis}.
\end{eqnarray}
Now setting $r=2(2-p^{-1})$, where $p$ is the exponent of inequality (\ref{ISI}) we achieve by H\"older inequality
\begin{eqnarray}
\big(\int |f|^r(\phis)^2\,dx\big)^{2/r}\leq CQ^{\phis}[f],\ \forall\,f\in\calC^{\phis},
\end{eqnarray}
yielding a Sobolev embedding for $D(Q^{\phis})$, which in turns yields the ultracontractivity of $T_t^{\phis}$, which was to be proved.
\end{proof}
The latter theorem  leads directly to
\begin{coro}\begin{itemize}\item[i)] The large time asymptotic for $p_t^*$:
\begin{eqnarray}
p_t^*(x,y)\sim e^{-\lam_0^*t}\delta^{\alpha/2}(x)\delta^{\alpha/2}(y)w_*(x)w_*(y),\ a.e.\ {\rm for\ large}\ t.
\end{eqnarray}
\item[ii)] Green kernel lower bound
\begin{eqnarray}
G_*(x,y)\geq C\delta^{\alpha/2}(x)\delta^{\alpha/2}(y)w_*(x)w_*(y),\ a.e.
\end{eqnarray}
\item[iii)] Let $f\in L^2,\ f\geq 0\ a.e.$ and $u$ the weak solution of $L_*u=f$. Then
\begin{eqnarray}
u\geq C\delta^{\alpha/2}w_*,\ a.e.
\end{eqnarray}
\end{itemize}
\end{coro}

\begin{tabular}{ll}
A.~Beldi &  \quad N.~Belhadjrhouma\\
Faculty of Sciences of Tunis, Tunisia &  \quad Faculty of Sciences of Tunis, Tunisia\\
{\small beldiali@gmail.com} &  \quad
{\small nedra.belhadjrhouma@fst.rnu.tn} \\
& \\
& \\
A.~BenAmor & \\
 Faculty of Sciences of Gab\`es, Tunisia & \\
{\small ali.benamor@ipeit.rnu.tn } &
\end{tabular}

\bibliographystyle{alpha}
\bibliography{biblio-frac-V1}

\begin{thebibliography}{BAB11}

\bibitem[BAB11]{bbb}
Hichem BelHadjAli, Ali~Ben Amor, and Johannes~F. Brasche.
\newblock Large coupling convergence: Overview and new results.
\newblock In I.~Gohberg, Michael Demuth, Bert-Wolfgang Schulze, and Ingo Witt,
  editors, {\em Partial Differential Equations and Spectral Theory}, volume 211
  of {\em Operator Theory: Advances and Applications}, pages 73--117. Springer
  Basel, 2011.

\bibitem[Ba{\~n}91]{banuelos91}
Rodrigo Ba{\~n}uelos.
\newblock Intrinsic ultracontractivity and eigenfunction estimates for
  {S}chr\"odinger operators.
\newblock {\em J. Funct. Anal.}, 100(1):181--206, 1991.

\bibitem[BB12]{benamor-forum2012}
Nedra Belhajrhouma and Ali BenAmor.
\newblock {H}ardy's inequality in the scope of {D}irichlet forms.
\newblock {\em Forum Math.}, 24(4):751--767, 2012.

\bibitem[BBB]{benamor-osaka}
Ali Beldi, Nedra Belhajrhouma, and Ali BenAmor.
\newblock Pointwise estimates for the ground states of some classes of
  positivity preserving operators.
\newblock {\em Osaka J. Math.}, To appear.

\bibitem[BBC03]{bogdan}
Krzysztof Bogdan, Krzysztof Burdzy, and Zhen-Qing Chen.
\newblock Censored stable processes.
\newblock {\em Probab. Theory Related Fields}, 127(1):89--152, 2003.

\bibitem[BH86]{bliedtner}
J.~Bliedtner and W.~Hansen.
\newblock {\em Potential theory}.
\newblock Universitext. Springer-Verlag, Berlin, 1986.
\newblock An analytic and probabilistic approach to balayage.

\bibitem[CG98]{cipriani-grillo}
F.~Cipriani and G.~Grillo.
\newblock Pointwise properties of eigenfunctions and heat kernels of
  {D}irichlet-{S}chr\"odinger operators.
\newblock {\em Potential Anal.}, 8(2):101--126, 1998.

\bibitem[CKS10]{kim2010}
Zhen-Qing Chen, Panki Kim, and Renming Song.
\newblock Heat kernel estimates for the {D}irichlet fractional {L}aplacian.
\newblock {\em J. Eur. Math. Soc. (JEMS)}, 12(5):1307--1329, 2010.

\bibitem[CS03]{song}
Zhen-Qing Chen and Renming Song.
\newblock {H}ardy inequality for censored stable processes.
\newblock {\em Tohuku Math.J.}, 55, 2003.

\bibitem[Dav89]{davies-book}
Eduard~B. Davies.
\newblock {\em Heat kernels and spectral theory}.
\newblock Cambridge University Press, Cambridge, 1989.

\bibitem[DD03]{davilla-dupaigne}
Juan D{\'a}vila and Louis Dupaigne.
\newblock Comparison results for {PDE}s with a singular potential.
\newblock {\em Proc. Roy. Soc. Edinburgh Sect. A}, 133(1):61--83, 2003.

\bibitem[DN02]{dupaigne-nedev}
Louis Dupaigne and Gueorgui Nedev.
\newblock Semilinear elliptic {PDE}'s with a singular potential.
\newblock {\em Adv. Differential Equations}, 7(8):973--1002, 2002.

\bibitem[Fit00]{fitz00}
P.J. Fitzsimmons.
\newblock {Hardy's inequality for Dirichlet forms.}
\newblock {\em J. Math. Anal. Appl.}, 250(2):548--560, 2000.

\bibitem[FLS08]{frank}
Rupert~L. Frank, Elliott~H. Lieb, and Robert Seiringer.
\newblock Hardy-{L}ieb-{T}hirring inequalities for fractional {S}chr\"odinger
  operators.
\newblock {\em J. Amer. Math. Soc.}, 21(4):925--950, 2008.

\bibitem[F{\=O}T94]{fuku-oshima}
Masatoshi Fukushima, Y{\=o}ichi {\=O}shima, and Masayoshi Takeda.
\newblock {\em Dirichlet forms and symmetric {M}arkov processes}.
\newblock Walter de Gruyter \& Co., Berlin, 1994.

\bibitem[Han06]{Hansen-comp}
Wolfhard Hansen.
\newblock Global comparison of perturbed {G}reen functions.
\newblock {\em Math. Ann.}, 334(3):643--678, 2006.

\bibitem[Kul98]{kul}
Tadeusz Kulczycki.
\newblock Intrinsic ultracontractivity for symmetric stable processes.
\newblock {\em Bull. Polish Acad. Sci. Math.}, 46(3):325--334, 1998.

\bibitem[MR92]{roeckner-ma}
Zhi~Ming Ma and Michael R{\"o}ckner.
\newblock {\em Introduction to the theory of (nonsymmetric) {D}irichlet forms}.
\newblock Springer-Verlag, Berlin, 1992.

\bibitem[SC02]{saloff-coste}
Laurent Saloff-Coste.
\newblock {\em Aspects of {S}obolev-type inequalities}, volume 289 of {\em
  London Mathematical Society Lecture Note Series}.
\newblock Cambridge University Press, Cambridge, 2002.

\bibitem[VZ00]{zuazua}
Juan~Luis Vazquez and Enrike Zuazua.
\newblock The {H}ardy inequality and the asymptotic behaviour of the heat
  equation with an inverse-square potential.
\newblock {\em J. Funct. Anal.}, 173(1):103--153, 2000.

\bibitem[Yaf99]{yafaev}
D.~Yafaev.
\newblock Sharp constants in the {H}ardy-{R}ellich inequalities.
\newblock {\em J. Funct. Anal.}, 168(1):121--144, 1999.

\end{thebibliography}

\end{document}